\documentclass[10pt, leqno, a4paper]{amsart}

\usepackage[latin1]{inputenc}
\usepackage{amsfonts}
\usepackage{amsmath}
\usepackage{amssymb}
\usepackage{amsthm}
\usepackage[T1]{fontenc}
\usepackage[dvips]{graphicx}
\usepackage{enumerate}
\usepackage{verbatim}

\usepackage[dvips]{hyperref}
\usepackage[matrix, arrow]{xy}

\DeclareMathOperator{\Proj}{Proj}
\DeclareMathOperator{\Hom}{Hom}
\DeclareMathOperator{\Syz}{Syz}
\DeclareMathOperator{\rk}{rk}

\DeclareMathOperator{\Ext}{Ext}

\DeclareMathOperator{\chara}{char}
\DeclareMathOperator{\Spec}{Spec}

\DeclareMathOperator{\im}{im}

\newcommand{\eps}{\varepsilon}

\input xy
\xyoption{all}
\SelectTips{cm}{10}

\begin{document}
\swapnumbers
\theoremstyle{definition}
\newtheorem{Le}{Lemma}[section]
\newtheorem{Def}[Le]{Definition}
\newtheorem*{DefB}{Definition}
\newtheorem{Bem}[Le]{Remark}
\newtheorem{Ko}[Le]{Corollary}
\newtheorem{Theo}[Le]{Theorem}
\newtheorem*{TheoB}{Theorem}
\newtheorem{Bsp}[Le]{Example}
\newtheorem{Be}[Le]{Observation}
\newtheorem{Prop}[Le]{Proposition}
\newtheorem{Sit}[Le]{Situation}
\newtheorem{Que}[Le]{Question}
\newtheorem*{Con}{Conjecture}
\newtheorem{Dis}[Le]{Discussion}
\newtheorem{Prob}[Le]{Problem}
\newtheorem{Konv}[Le]{Convention}

\title[An inclusion result for dagger closure]{An inclusion result for dagger closure in certain section rings of abelian varieties}
\author{Axel St\"abler}
\address{Universit\"at Osnabr\"uck, Fachbereich 6: Mathematik/Informatik,
Albrechtstr. 28a,
49069 Osnabr\"uck, Germany}
\email{axel.staebler@uni-osnabrueck.de}
\curraddr{Johannes Gutenberg-Universi\"at Mainz\\ Fachbereich 08\\
Staudingerweg 9\\
55128 Mainz\\
Germany}

\begin{abstract}
We prove an inclusion result for graded dagger closure for primary ideals in symmetric section rings of abelian varieties over an algebraically closed field of arbitrary characteristic.
\end{abstract}
\maketitle

\section*{Introduction}
Dagger closure was introduced in \cite{hochsterhunekedagger} by Hochster and Huneke as an alternative characteristic-free characterisation of tight closure (in terms of valuations) in complete local domains of positive prime characteristic. Heitmann used a variant of dagger closure to prove the direct summand conjecture in mixed characteristic in dimension $3$ (see \cite{heitmanndirectsummand}).

In \cite{brennerstaeblerdaggersolid}, Brenner and the author introduced a graded variant of dagger closure and proved, using geometric methods, that this coincides with solid closure (see \cite{hochstersolid} for background on this closure operation) for homogeneous ideals in an $\mathbb{N}$-graded domain $R$ of finite type over a field $R_0$ in dimension $2$. And in \cite{brennerstaeblerdaggerregular} Brenner and the author proved that graded dagger closure is trivial for polynomial rings over a field.

We recall the definition of graded dagger closure. The absolute integral closure $R^+$ is the integral closure of a domain $R$ in an algebraic closure of $Q(R)$ and $R^{+ \text{GR}}$ is the maximal $\mathbb{Q}_{\geq 0}$-graded subring of $R^+$ which extends the grading given by $R$ (see \cite[Lemma 4.1]{hochsterhunekeinfinitebig} for a detailed exposition).

\begin{DefB}
Let $R$ denote an $\mathbb{N}$-graded domain and let $I$ be an ideal of $R$. Let $\nu$ be the valuation on $R^{+ \text{GR}}$ induced by the grading. Then we define the \emph{graded dagger closure} $I^{\dagger \text{GR}}$ of an ideal $I$ as the set of elements $f$ in $R$ such that for all  $\varepsilon > 0$ there exists an element $a \in R^{+ \text{GR}}$ with $\nu(a) < \varepsilon$ and such that $af$ lies in the extended ideal $IR^{+ \text{GR}}$.
\end{DefB}

In this paper we will prove an inclusion bound for dagger closure for primary ideals in section rings of abelian varieties over an algebraically closed field of arbitrary characteristic associated to symmetric line bundles (see Theorem \ref{AbelianInclusion}). This also provides examples where graded dagger closure is non-trivial in normal rings of dimension $\geq 2$ (see Example \ref{BspDaggerNontrivial}). This is in particular interesting since graded dagger closure may be thought of as an asymptotic variant of graded plus closure (an element $f$ is in the graded plus closure $I^{+ \text{gr}}$ if $f \in IR^{+ \text{GR}}$), and graded plus closure is trivial in equal characteristic zero for normal rings. We will also relax the conditions on a similar theorem for standard-graded normal Cohen-Macaulay domains over an algebraically closed field of positive characteristic for tight closure which was proven by Brenner (see \cite{brennerlinearfrobenius}).

The key tool in positive characteristic is of course the (absolute) Frobenius morphism. The crucial property that one needs is that the Frobenius can manipulate the degree of a vector bundle for a given polarisation while having the same source and target, thereby fixing all invariants of the variety. Similarly, the multiplication by $N$ map on an abelian variety $X$ is a finite surjective morphism from $X$ to $X$ of degree $N^{2 \dim X}$.

There are however some difficulties. For instance, Brenner's proof requires that the ring $R$ is Cohen-Macaulay.
But the section ring of an abelian variety $X$ is not Cohen-Macaulay if $\dim X \geq 2$ (see Proposition \ref{AbelianNotACM} below). However, the deviation from being Cohen-Macaulay only comes from the intermediate cohomology of $\mathcal{O}_X$. Hence, tensoring by a suitable root of the polarising line bundle this difficulty may be circumvented.

This article is based on parts of the author's Ph.\,D.\ thesis (\cite{staeblerthesis}). I thank Holger Brenner for useful discussions.

\section{Some preliminaries}
In this section we review some facts about slope semistability and abelian varieties.
Let $(X, \mathcal{O}_X(1))$ be a polarised\footnote{By which we mean that $\mathcal{O}_X(1)$ is an ample invertible sheaf.} projective variety of dimension $d$ over an algebraically closed field $k$. The degree (with respect to $\mathcal{O}_X(1)$) of a coherent torsion free sheaf $\mathcal{F}$ is defined as $\deg \mathcal{F} = \alpha_{d-1}(\mathcal{E}) - \rk \mathcal{E} \cdot \alpha_{d-1}(\mathcal{O}_X)$, where $\alpha_i(\mathcal{E})$ are the coefficients of the Hilbert polynomial of $\mathcal{E}$. If $X$ is smooth then this is the same as $\det (\mathcal{F}).\mathcal{O}_X(1)^{d-1}$.

The \emph{slope} of $\mathcal{F}$ (with respect to $\mathcal{O}_X(1)$) is defined as $\mu(\mathcal{F}) = \deg \mathcal{F}/\rk \mathcal{F}$. The \emph{minimal slope} of $\mathcal{F}$, denoted by $\mu_{\min}(\mathcal{F})$ and the \emph{maximal slope}, written $\mu_{\max}(\mathcal{F})$, are defined as \[\min \{\mu(\mathcal{Q}) \, \vert \, \mathcal{F} \to \mathcal{Q} \to 0 \text{ is a torsion-free quotient sheaf of rank } \geq 1 \}\] and \[ \max \{\mu(\mathcal{E}) \, \vert \, 0 \to \mathcal{E} \to \mathcal{F} \text{ is a subsheaf of rank }\geq 1 \}\]respectively.

If $f: X' \to X$ is a finite dominant separable morphism of normal projective varieties over $k$ then we have $\mu(f^\ast\mathcal{S}) = \deg f \cdot \mu(\mathcal{S})$ for any locally free sheaf $\mathcal{S}$ on $X$, where $\deg f$ is the degree of the field extension in the generic point and the polarisation on $X'$ is given by $f^\ast \mathcal{O}_X(1)$ (see \cite[proof of Lemma 3.2.2]{huybrechtslehn}). Also $\mu_{\min}$ and $\mu_{\max}$ transform in the same way.
A locally free sheaf $\mathcal{S}$ is said to be \emph{semistable} if $\mu(\mathcal{F}) \leq \mu(\mathcal{S})$ for all coherent subsheaves $\mathcal{F}$ of $\mathcal{S}$. A locally free sheaf is called \emph{strongly semistable} if $F^{e \ast}\mathcal{S}$ is semistable for all $e \geq 0$, where $F$ denotes the (relative) Frobenius morphism. In characteristic zero we consider any semistable sheaf to be strongly semistable.

Moreover, we define $\bar{\mu}_{\min}(\mathcal{S})$  as \[\min \{\frac{\mu_{\min}(\varphi^\ast\mathcal{S})}{\deg \varphi} \, \vert \, \varphi: X' \to X \text{ is finite dominant}\} \] and one similarly defines $\bar{\mu}_{\max}(\mathcal{S})$. 

Crucial for us will be that if $\mu_{\min}(\mathcal{F}) > \deg \mathcal{L}$ then $\Hom(\mathcal{F}, \mathcal{L}) = 0$, where $\mathcal{L}$ is a line bundle. This is well-known but since the author is unable to provide a suitable reference we shall give a proof.

\begin{Le}
\label{muMaxMin}
Let $(X, \mathcal{O}_X(1))$ be a polarised projective variety and $\mathcal{E}$, $\mathcal{F}$ coherent torsion-free sheaves of $\rk \geq 1$. If $\mu_{\min}(\mathcal{E}) > \mu_{\max}(\mathcal{F})$ then $\Hom(\mathcal{E}, \mathcal{F}) = 0$.
\end{Le}
\begin{proof}
Let $\varphi: \mathcal{E} \to \mathcal{F}$ be a morphism and denote the coherent image sheaf by $\mathcal{I}$. Assume that $\varphi \neq 0$ and hence $\mathcal{I} \neq 0$. Then we necessarily have $\rk \mathcal{I} \geq 1$ since $\mathcal{F}$ is torsion-free. Therefore, $\mu_{\min}(\mathcal{E}) \leq \mu(\mathcal{I})$ and $\mu(\mathcal{I}) \leq \mu_{\max}(\mathcal{F})$. That is, $\mu_{\min}(\mathcal{E}) \leq \mu_{\max}(\mathcal{F})$.
\end{proof}

We also note the following variant.

\begin{Le}
\label{CohoSerreVanishing}
Let $(X, \mathcal{O}_X(1))$ be a polarised smooth projective variety of dimension $d$ over an algebraically closed field and $\mathcal{S}$ a locally free sheaf on $X$. If $\mu_{\min}(\mathcal{S}) > \deg \omega_X$ then $H^d(X, \mathcal{S}) = 0$.
\end{Le}
\begin{proof}
Serre duality yields that $H^d(X, \mathcal{S}) = \Ext^0(\mathcal{S}, \omega_X)^\vee$. Since $\mu_{\min}(\mathcal{S}) > \deg \omega_X$ the claim follows from Lemma \ref{muMaxMin}.
\end{proof}

Next we collect some results on abelian varieties. All statements may be found in Mumford's book \cite{mumfordabelian}.
An abelian variety $X$ is a complete variety over an algebraically closed field $k$ with a group structure $\mu: X \times_k X \to X$ such that $\mu$ and the inverse are morphisms. The group structure is then necessarily commutative and $X$ is smooth and projective.

As stated earlier our replacement for the (relative) Frobenius will be the multiplication by $N$ map $N_X: X \to X$ whose properties we summarize in the following

\begin{Prop}
\label{AbelianMult}
Let $X$ be an abelian variety and $N$ an integer with $\chara k \nmid N$. Then the multiplication map $N_X: X \to X$ is a finite surjective separable morphism of degree $N^{2 \dim X}$. Furthermore, for all $N \in \mathbb{Z}$ and $\mathcal{L}$ a line bundle we have $N_X^\ast \mathcal{L} = \mathcal{L}^{\frac{N^2+N}{2}} \otimes (-1_X)^\ast \mathcal{L}^{\frac{N^2 -N}{2}}$.
\end{Prop}
\begin{proof}
The surjectivity is proved in \cite[(iv) after Question II.4.4]{mumfordabelian} and the separability and the claim about the degree are treated in [ibid., Application II.6.3]. In order to prove that $N_X$ is finite note that $N_X = \mu^{N-1} \circ \Delta_{N-1}$, where $\Delta_{N-1}$ is the $N-1$-fold diagonal morphism. Since these are both projective (\cite[Proposition 5.5.5 (i) and (v)]{EGAII}) and $N_X$ has finite fibres it is finite by \cite[Ex. III.11.2]{hartshornealgebraic}. 
For the second assertion see \cite[Corollary II.6.3]{mumfordabelian}.
\end{proof}

A line bundle $\mathcal{L}$ on an abelian variety $X$ is called \emph{symmetric} if $\mathcal{L} = (-1)_X^\ast \mathcal{L}$. If $\mathcal{L}$ is an ample line bundle then $\mathcal{M} = \mathcal{L} \otimes (-1)_X^\ast \mathcal{L}$ is ample and symmetric and $N_X^\ast \mathcal{M} = {\mathcal{M}}^{N^2}$.

Finally, we note some properties about the cohomology of line bundles on an abelian variety

\begin{Prop}
\label{AbelianProperties}
Let $X$ be an abelian variety and $\mathcal{L}$ an ample invertible sheaf on $X$. Then the following hold:
\begin{enumerate}[(a)]
\item{$\dim H^i(X, \mathcal{O}_X) = \binom{\dim X}{i}$.}
\item{There exists an $0 \leq i = i(\mathcal{L}) \leq \dim X$ such that $H^i(X, \mathcal{L}) \neq 0$ and $H^j(X, \mathcal{L}) = 0$ for $j \neq i$. Moreover, $H^l(X, \mathcal{L}^{-1}) \neq 0$ if and only if $l = \dim X - i(\mathcal{L})$.}
\item{For any $n \geq 3$, $\mathcal{L}^n$ is very ample.}
\end{enumerate}
\end{Prop}
\begin{proof}
All references are to \cite{mumfordabelian}.
For (a) see [Corollary III.13.2]. Part (b) is the Vanishing Theorem in [III.16]. The assertion in (c) is proved in [III.17].
\end{proof}

\section{Almost zero for vector bundles on projective varieties}

In this section we generalise parts of the theory of almost zero as developed in \cite{brennerstaeblerdaggersolid} for curves to higher dimensional varieties which we will then use to prove the inclusion result in the next section.
Throughout this section we work over a fixed algebraically closed field of arbitrary characteristic.
Recall that a line bundle $\mathcal{L}$ on a scheme is called semiample if $\mathcal{L}^m$ is globally generated for some $m > 0$.

\begin{Def}
\label{AlmostZeroDefGeneral}
Let $(Y, \mathcal{O}_Y(1))$ be a polarised projective variety over an algebraically closed field and $\mathcal{S}$ a locally free sheaf on $Y$ together with a cohomology class $c \in H^1(Y, \mathcal{S})$. We say that $c$ is \emph{almost zero} if for all $\eps > 0$ there exists a finite dominant morphism $\varphi: Y' \to Y$ between projective varieties and a semiample line bundle $\mathcal{L}$ on $Y'$ with a global section $s \neq 0$ such that $\deg \mathcal{L}/\deg \varphi < \eps$ and such that $s \varphi^\ast(c) = 0$ in $H^1(Y', \mathcal{L} \otimes \varphi^\ast \mathcal{S})$. Here $s \varphi^\ast(c)$ is induced by the morphism \[\begin{xy}\xymatrix{0 \ar[r]& \varphi^\ast \mathcal{S} \ar[r]^>>>>>{\cdot s}& \varphi^\ast \mathcal{S} \otimes \mathcal{L}}, \end{xy}\] 
the degree on $Y'$ is with respect to $\varphi^\ast \mathcal{O}_Y(1)$ and $\deg \varphi$ is the degree of the function field extension.
We say that $\mathcal{S}$ is almost zero if every $c \in H^1(Y, \mathcal{S})$ is almost zero.
\end{Def}

\begin{Bem}
\label{Bemallgaz}
The situation of interest from the view point of graded dagger closure is $Y = \Proj R$, where $R$ is an $\mathbb{N}$-graded domain finitely generated over an algebraically closed field $R_0$. Since we may always pass to finite ring extensions for graded dagger closure we may assume that $R$ is normal and that $Y$ is covered by standard open sets coming from elements of degree $1$. 
Hence, $R(1)^\sim = \mathcal{O}_Y(1)$ is an ample invertible sheaf which is generated by global sections. Therefore, in this case there is a canonical choice for a polarisation.

This definition agrees with \cite[Definition 4.1]{brennerstaeblerdaggersolid} if $\dim X = 1$. To begin with, we may assume $X$ to be normal. Moreover, \cite[Lemma IV.1.2]{hartshornealgebraic} and the fact that a line bundle is ample if and only if its degree is positive ([ibid., Corollary IV.3.3]) yield that semiample line bundles with a global section are precisely those of non-negative degree with a global section.

Also note that the line bundles in Definition \ref{AlmostZeroDefGeneral} are of degree $\geq 0$ by the Nakai Criterion (cf. \cite[Theorem I.5.1]{hartshorneamplesubvarieties}) and since they are effective. 
\end{Bem}

\begin{Bem}
\label{Observation}
In the situation of Definition \ref{AlmostZeroDefGeneral} a section $s \in H^0(Y', \mathcal{L})$ annihilating $c$ corresponds to a morphism $\mathcal{L}^\vee \to \varphi^\ast \mathcal{S}'$, where $\mathcal{S}'$ denotes the extension of $\mathcal{O}_Y$ by $\mathcal{S}$ induced by $c$, which does not factor through $\varphi^\ast \mathcal{S}$.

Indeed, let $\varphi: Y' \to Y$ be a finite morphism of projective varieties and $\mathcal{L}$ a semiample line bundle on $Y'$ with a nonzero global section $s$ such that $sc \in H^1(\varphi^\ast\mathcal{S} \otimes \mathcal{L})$ is zero. Looking at the cohomology of the short exact sequence \[\begin{xy}\xymatrix{0 \ar[r]& \varphi^\ast \mathcal{S} \otimes \mathcal{L} \ar[r]& \varphi^\ast \mathcal{S}' \otimes \mathcal{L} \ar[r]& \mathcal{L} \ar[r]& 0}\end{xy}\] we obtain a nonzero global section $H^0(Y', \varphi^\ast \mathcal{S}' \otimes \mathcal{L})$ which corresponds to a morphism $\mathcal{L^\vee} \to \varphi^\ast \mathcal{S}'$ (since $s$ maps to $sc =0$ along $H^0(Y', \mathcal{L}) \to H^1(Y', \varphi^\ast \mathcal{S} \otimes \mathcal{L})$ it induces the desired global section). Moreover, as $s$ is nonzero it is not contained in $H^0(Y', \varphi^\ast \mathcal{S} \otimes \mathcal{L})$. Hence, the morphism does not factor through $\varphi^\ast \mathcal{S}$.  
\end{Bem}

\begin{Prop}
\label{PullbackSemiampleGloballygenerated}
Let $Y$ be a normal projective variety over an algebraically closed field and let $\mathcal{L}$ be a semiample line bundle on $Y$. Then there is a finite dominant morphism $\varphi: Y' \to Y$ of normal projective varieties such that $\varphi^\ast \mathcal{L}$ is generated by global sections. Moreover, if $Y$ is smooth and if $\mathcal{L}$ is ample then $Y'$ may be chosen to be smooth as well.
\end{Prop}
\begin{proof}
Since the proof of \cite[Lemma 5.2]{brennerstaeblerdaggersolid} also works in higher dimensions one obtains this similarly to \cite[Corollary 5.3]{brennerstaeblerdaggersolid}. If $\mathcal{L}$ is ample we can fix $m \gg 0$ such that $\mathcal{L}^m$ is very ample and invoke \cite[Corollary 1.15.1]{maruyamagrauertmuelich} with $\mathcal{L}^m = \mathcal{O}_X(1)$ and $d = m$.
\end{proof}

\begin{Prop}
\label{Varazroots}
Let $R$ be an $\mathbb{N}$-graded normal domain finitely generated over an algebraically closed field $R_0 =k$ of dimension $d \geq 2$. Furthermore, assume that $\Proj R$ is covered by open sets $D_+(g)$, $g \in R_1$, and let $I = (f_1, \ldots, f_n)$ be a homogeneous $R_+$-primary ideal. Fix a homogeneous element $f_0$ of degree $d_0$ and write $Y = \Proj R$ and $\mathcal{S} = \Syz(f_1, \ldots, f_n)(d_0)$. Then $f_0$ is contained in the graded dagger closure of $I$ if and only if $c = \delta(f_0) \in H^1(Y,\mathcal{S})$ is almost zero and we can choose the annihilating line bundles as roots of $\mathcal{O}_Y(1)$.
\end{Prop}
\begin{proof}
If $f_0 \in I^{\dagger \text{GR}}$ then exactly the same argument as in the first paragraph of the proof of \cite[Theorem 5.6]{brennerstaeblerdaggersolid}  shows that $\delta(f_0) = c$ is almost zero and that the annihilating line bundles can be chosen as roots of $\mathcal{O}_Y(1)$.

Conversely, let $\eps > 0$. 
Assume that there is a finite dominant morphism $\varphi: X \to Y$ and an ample line bundle $\mathcal{L}$ on $X$ such that $\mathcal{L}^m = \varphi^\ast \mathcal{O}_Y(1)$ for some $m$ together with a nonzero global section $s \in H^0(X, \mathcal{L})$ such that the map $H^1(X, \varphi^\ast \mathcal{S}) \xrightarrow{\cdot s} H^1(X, \varphi^\ast \mathcal{S}\otimes \mathcal{L})$ annihilates $c$, where $\deg \mathcal{L}/\deg \varphi < \eps$.
By virtue of Proposition \ref{PullbackSemiampleGloballygenerated}, we may assume that $\mathcal{L}$ is generated by global sections and passing to the normalisation we may assume $X$ to be normal. 

Consider the section ring $S = \bigoplus_{n \geq 0} \Gamma(X, \mathcal{L}^n)$, where we fix a minimal $m$ such that $\mathcal{L}^m$ is isomorphic to $\varphi^\ast \mathcal{O}_Y(1)$ and identify these line bundles along a fixed isomorphism. This is then a finite normal graded extension domain of $R$ in light of \cite[Proposition 3.6]{brennerstaeblerdaggersolid} and \cite[Proposition 2.1]{hyrysmith}. Since we have $m \deg \mathcal{L} = \deg \varphi^\ast \mathcal{O}_Y(1)$, elements in $H^0(X, \mathcal{L})$ correspond to homogeneous elements of degree $\frac{1}{m}$ in $S$. Moreover, we have $\frac{1}{m} \leq \deg \mathcal{L}/\deg \varphi < \eps$ since $\deg \mathcal{L}^m/\deg \varphi \geq 1$.

Now consider the exact sequence \[\begin{xy} \xymatrix{0 \ar[r]& \varphi^\ast \mathcal{S} \ar[r]& \varphi^\ast \mathcal{S}' \ar[r]& \mathcal{O}_X \ar[r]& 0,}\end{xy}\] where $\mathcal{S}' = \Syz(f_0, f_1, \ldots, f_n)(d_0)$ is the extension given by $c$, tensor with $\mathcal{L}$ and take cohomology. This yields that $s \in H^0(X, \mathcal{L})$ has a nonzero preimage $r \in H^0(X, \varphi^\ast \mathcal{S}' \otimes \mathcal{L})$ which is not contained in $H^0(X, \varphi^\ast \mathcal{S} \otimes \mathcal{L})$.

The section $r$ therefore corresponds to a relation $a_0 f_0 = \sum_i a_i f_i$ in $S$ and since it is not contained in $\varphi^\ast \mathcal{S}$ we must have $a_0 \neq 0$ and $a_0$ has degree $\frac{1}{m} < \eps$.
\end{proof}

\begin{Bem}
\begin{enumerate}[(a)]
\item{

Of course, we would very much like to drop the assumption that the annihilating line bundles be roots of $\mathcal{O}_Y(1)$. But we do not know whether this is possible. The critical point where the proof of \cite[Theorem 5.6]{brennerstaeblerdaggersolid} fails to work is that it is not enough that $\deg \mathcal{L}^\vee(t) > 0$ any more. This ensured that the line bundle was ample so that it was globally generated after some finite dominant pullback. In general, we need that $\mathcal{L}^\vee(t)$ is generated by global sections (after some finite dominant pullback) while at the same time bounding $\deg \mathcal{L}^\vee(t)/\deg \varphi$. Of course, there is some power $t$ such that $\mathcal{L}^\vee(t)$ is generated by global sections but we cannot control $t$ while bounding the quotient -- not even on a curve.

The ampleness property is most likely too restrictive in higher dimensions. Ideally one would like to have a property on the annihilating line bundle $\mathcal{L}$ such that a ``small'' twist of $\mathcal{L}^\vee$ by a suitable root of $\mathcal{O}_Y(1)$ is semiample.}
\item{Assume that $R$ is an $\mathbb{N}$-graded normal domain over an algebraically closed field $R_0$, $\dim R \geq 2$ and that $\mathcal{O}_X(1) = R(1)^\sim$ on $X = \Proj R$ is invertible. Furthermore, let $(f_1, \ldots, f_n)$ be a homogeneous $R_+$-primary ideal with homogeneous generators $f_i$ and let $f_0 \in R_{d_0}$.

If the characteristic of $R_0$ is positive then the issue whether $\delta(f_0)$ is almost zero in $H^1(X, \Syz(f_1, \ldots, f_n)(d_0))$ with respect to roots of $\mathcal{O}_X(1)$ is equivalent to whether the torsor $T$ associated to $\delta(f_0)$ has cohomological dimension $d - 1$. And in characteristic zero one still has that the cohomological dimension of $T$ is $d-1$ if $\delta(f_0)$ is almost zero with respect to roots of $\mathcal{O}_X(1)$.

This follows via a geometric interpretation for another closure operation -- the so-called solid closure (see \cite[Proposition 3.9]{brennertightproj}). By \cite[Corollary 4.7]{brennerstaeblerdaggerregular} dagger closure is contained in solid closure and they both coincide with tight closure in positive characteristic by \cite[Corollary 2.12]{brennerstaeblerdaggerregular} and \cite[Paragraph 8]{hochstersolid}.

It would be very interesting to have geometric proofs of these facts. In particular, having a geometric proof in characteristic zero might yield an alternative characterisation of almost zero with respect to roots of $R(1)^\sim$.}
\end{enumerate}
\end{Bem}

\begin{Prop}
\label{PropExclusionBound}
Let $Y$ be a polarised projective normal variety over an algebraically closed field $k$ and $\mathcal{S}$ a locally free sheaf on $Y$. If $c \in H^1(Y, \mathcal{S}), c \neq 0$ is almost zero then $\bar{\mu}_{\max}(\mathcal{S}') \geq 0$, where $\mathcal{S}'$ denotes the extension of $\mathcal{O}_Y$ by $\mathcal{S}$ induced by $c$. 
\end{Prop}
\begin{proof}
Let $\varphi: Y' \to Y$ be a finite morphism of projective varieties and $\mathcal{L}$ a semiample line bundle on $Y'$ with a nonzero global section $s$ such that $sc \in H^1(\varphi^\ast\mathcal{S} \otimes \mathcal{L})$ is zero.

By Remark \ref{Observation}, we obtain a nontrivial morphism $\mathcal{L}^\vee \to \varphi^\ast \mathcal{S}'$. Hence, by Lemma \ref{muMaxMin} we obtain $\deg \mathcal{L}^\vee \leq \mu_{\max}(\varphi^\ast \mathcal{S}')$. Since $\deg \varphi \cdot \bar{\mu}_{\max}(\mathcal{S}') \geq \mu_{\max}(\varphi^\ast \mathcal{S}')$ we obtain the desired inequality. Indeed, we have $\deg \mathcal{L}^\vee/ \deg \varphi \leq \bar{\mu}_{\max}(\mathcal{S}')$ and since $c$ is assumed to be almost zero we find for $\eps >0$ a $\varphi:Y' \to Y$ and $\mathcal{L}$ with $\deg \mathcal{L} / \deg \varphi < \eps$.
\end{proof}

An immediate consequence is

\begin{Ko}
\label{KSchrott}
Let $R$ be an $\mathbb{N}$-graded domain such that $\mathcal{O}_Y(1) = R(1)^\sim$ on $Y = \Proj R$ is invertible. Furthermore, let $I$ be a homogeneous $R_+$-primary ideal with homogeneous generators $f_1, \ldots, f_n$. If $f_0 \in (f_1, \ldots, f_n)^{\dagger \text{GR}}$ and $f_0 \notin I$ then $\bar{\mu}_{\max}(\Syz(f_0, \ldots, f_n)) \geq 0$.
\end{Ko}

\begin{Bem}
We do not think that the result of Corollary \ref{KSchrott} actually provides a useful exclusion bound if $\dim R = d + 1 \geq 3$. The heuristic here is that in higher dimensions the first syzygy bundle has to be ``very positive'' if one wants to have nontrivial containment relations. It seems that in this case one should look at the $d$th cohomology of the $d$th syzygy bundle in a resolution on $\Proj R$.

Consider $R = k[x,y,z]$ and $I = (x^a, y^a, z^a)$, where $k$ is an algebraically closed field. Since $x^a, y^a, z^a$ form a regular sequence of parameters the Koszul complex is a free resolution of $R/I$.
  
Sheafifying this complex on $\Proj R = \mathbb{P}^2_k$ we obtain $\mathcal{S}_1 = \Syz(x^a, y^a, z^a)$ as the first syzygy bundle. The bundle $\mathcal{S}_1$ is strongly semistable of slope $\mu(\mathcal{S}_1) = -\frac{3}{2}a$ and the minimal $l$ such that $R_l \subseteq (x^a, y^a, z^a)^{\dagger \text{GR}}$ is $l = 3a -2$ (of course the polarisation is with respect to $\mathcal{O}_{\mathbb{P}^2_k}(1)$). This follows since $(x^a, y^a, z^a)^{\dagger \text{GR}}  = (x^a, y^a, z^a)$ by \cite[Corollary 3.9]{brennerstaeblerdaggerregular}. The top-dimensional syzygy bundle is given by $\mathcal{S}_d = \mathcal{O}_{\mathbb{P}^2_k}(-3a)$ which is of slope $\mu(\mathcal{S}_d) = -3a$, while $\mu(\mathcal{S}_1) = -\frac{3}{2}a$.

If $\chara k =0$ then using restriction theorems (e.\,g.\ \cite[Theorem 1.2]{flennerrestriction}) one has that $\mathcal{S}_1\vert_C$ is semistable on a suitable smooth curve $C \subset \mathbb{P}^2_k$ and the slope of $\mathcal{S}_1\vert_C$ is $- \frac{3}{2} a \deg C$. In particular, the minimal $l$ is in this case $\frac{3}{2} a$ (use \cite[Proposition 7.9]{brennerstaeblerdaggersolid}). Of course, on a curve $\mathcal{S}_1\vert_C$ is the top dimensional syzygy bundle.
\end{Bem}
\section{The inclusion result}

\begin{Def}
Let $X$ be a projective variety over an algebraically closed field and let $\mathcal{L}$ be an ample line bundle on $X$. The pair $(X, \mathcal{L})$ is called \emph{arithmetically Cohen-Macaulay} if the section ring $\bigoplus_{n \geq 0} \Gamma(X, \mathcal{L}^n)$ is Cohen-Macaulay.
\end{Def}

Note that our definition of arithmetically Cohen-Macaulay is slightly different from \cite[Definition 1.2.2]{miglioredeficiency} which seems to be the standard one. Our definition looks at the section ring of $\mathcal{L}$ which is, in the case that $\mathcal{L}$ is very ample and $X$ normal, the normalisation of the homogeneous coordinate ring of the embedding (see \cite[Ex.\ II.5.14]{hartshornealgebraic}).

\begin{Prop}
\label{AbelianNotACM}
Let $X$ be an abelian variety and $\mathcal{L}$ an ample invertible sheaf. Then $(X ,\mathcal{L})$ is arithmetically Cohen-Macaulay if and only if $\dim X = 1$, that is, $X$ is an elliptic curve.
\end{Prop}
\begin{proof}
The if-part follows from \cite[Proposition 2.1 (1)]{hyrysmith}. Conversely, the section ring $S$ is Cohen-Macaulay if and only if $H^i_{S_+}(S) = 0$ for $0 \leq i \leq \dim S -1$ (see \cite[Propositions 3.5.4, 3.6.4 and Theorem 3.6.3]{brunsherzog}) and by \cite[Theorem A4.1]{Eisenbud} we have $H^i(X, \mathcal{O}_X(n)) = H^{i+1}_{S_+}(S)_n$ for $i > 0$. Proposition \ref{AbelianProperties} (a) implies that $\mathcal{O}_X$ has non-vanishing intermediate cohomology if $\dim X \geq 2$. Hence, $S$ cannot be Cohen-Macaulay in this case.
\end{proof}

\begin{Le}
\label{KohoHoppingAb}
Let $X$ denote an abelian variety of dimension $d$ and $\mathcal{O}_X(1)$ a very ample line bundle on $X$. Let $\mathcal{S}$ be a locally free sheaf on $X$. Let
\[\begin{xy}\xymatrix{\ldots \ar[r]& \mathcal{G}_3 \ar[r]& \mathcal{G}_2 \ar[r]& \mathcal{S} \ar[r]& 0 }\end{xy}\] 
denote an exact complex of sheaves, where $\mathcal{G}_j$ has type $\bigoplus_{(k,j)} \mathcal{O}_Y(-\alpha_{k,j})$ with $\alpha_{k,j} \neq 0$. Set $\mathcal{S}_1 = \mathcal{S}$ and set $\mathcal{S}_j = \im (\mathcal{G}_{j+1} \to \mathcal{G}_j) = \ker (\mathcal{G}_j \to \mathcal{G}_{j-1})$ for $j \geq 2$. Fix $i \in \{1, \ldots, d\}$. Then there are isomorphisms $H^i(X, \mathcal{S}_1) = H^{i+1}(X, \mathcal{S}_2) = \ldots = H^{d-1}(X, \mathcal{S}_{d-i})$ and an inclusion $H^i(X, \mathcal{S}_1) \to H^d(X, \mathcal{S}_{d + i -1})$.
\end{Le}
\begin{proof}
Since $\mathcal{O}_X(1)$ is very ample the only non-vanishing cohomology group of $\mathcal{O}_X(n)$ is $H^0(X, \mathcal{O}_X(n))$ for $n > 0$, and $H^d(X, \mathcal{O}_X(n))$ for $n < 0$, by Proposition \ref{AbelianProperties} (b). Since all the $\alpha_{k,j}$ are nonzero this means that $H^i(X, \mathcal{G}_j) = 0$ for $j \geq 2$ and $i = 1, \ldots, d-1$.

Looking at the cohomology of the short exact sequences $0 \to \mathcal{S}_{j+1} \to \mathcal{G}_{j+1} \to \mathcal{S}_j \to 0$ we can thus extract isomorphisms
\[
 H^i(X, \mathcal{S}_j) = H^{i+1}(X, \mathcal{S}_{j+1}) \text{ for } i = 1, \ldots, d-2,
\]
and inclusions $H^{d-1}(X, \mathcal{S}_{j}) \to H^d(X, \mathcal{S}_{j+1})$.
In particular, we have isomorphisms $H^i(X, \mathcal{S}_1) = H^{i+1}(X, \mathcal{S}_2) = \ldots = H^{d-1}(X, \mathcal{S}_{d-i})$ and an inclusion $H^i(X, \mathcal{S}_1) \to H^{d}(X, \mathcal{S}_{d + 1-i})$.
\end{proof}

\begin{Le}
Let $(X, \mathcal{O}_X(1))$ be a polarised abelian variety of dimension $d$ and assume that $\mathcal{O}_X(1)$ is symmetric. Then for a locally free sheaf $\mathcal{S}$ on $X$ we have $\mu(N_X^\ast \mathcal{S}) = N^2 \mu(\mathcal{S})$, where both degrees are with respect to $\mathcal{O}_X(1)$.
\end{Le}
\begin{proof}
The slope of $N^\ast_X \mathcal{S}$ with respect to $N_X^\ast\mathcal{O}_X(1) = \mathcal{O}_X(N^2)$ is $N^{2 d} \mu(\mathcal{S})$. The degree  with respect to $\mathcal{O}_X(l)$ of a line bundle $\mathcal{L}$ is given by \[\mathcal{L}.\mathcal{O}_X(l)^{d-1} = l^{d-1} \mathcal{L}.\mathcal{O}_X(1)^{d-1}.\]
Hence, the slope with respect to $\mathcal{O}_X(1)$ is given by $N^{2 d}/N^{2d - 2} \mu(\mathcal{S}) = N^2 \mu(\mathcal{S})$.
\end{proof}

\begin{Theo}
\label{AbelianInclusion}
Let $R = \bigoplus_{n \geq 0}\Gamma(X,\mathcal{O}_X(n))$ be the section ring of a polarised abelian variety $(X, \mathcal{O}_X(1))$ over an algebraically closed field of arbitrary characteristic and $\dim R = d +1$ (i.\,e.\ $\dim X = d$). Assume moreover that $\mathcal{O}_X(1)$ is symmetric. Let $I$ be a homogeneous $R_+$-primary ideal and let
\[ \ldots \to F_2 = \bigoplus_{(k,2)}R(-\alpha_{k,2}) \to F_1 = \bigoplus_{(k,1)}R(-\alpha_{k,1}) \to R \to R/I \to 0\] be a homogeneous complex of graded $R$-modules that is exact on $D_+(R_+)$ (e.\,g.\ a graded resolution of $I$).
Let \[ \ldots \to \mathcal{G}_2 = \bigoplus_{(k,2)} \mathcal{O}_X(-\alpha_{k,2}) \to  \mathcal{G}_1 = \bigoplus_{(k,1)} \mathcal{O}_X(-\alpha_{k,1}) \to \mathcal{O}_X \to 0 \]
denote the corresponding exact complex of sheaves on $X$.
Denote by $\mathcal{S}_j = \ker(\mathcal{G}_j \to \mathcal{G}_{j-1})$ the kernel sheaves. Finally, let $\nu = - \mu_{\min}(\mathcal{S}_d)/\deg \mathcal{O}_X(1)$.

Then we have the inclusion $R_{\geq \nu} \subseteq I^{\dagger \text{GR}}$.
\end{Theo}
\begin{proof}
Note that the $\mathcal{S}_j$ are locally free sheaves.
Fix an integer $N \geq 3$ such that $N \nmid \chara k$. Pulling back along $N_X$ we may assume $\mathcal{O}_X(1)$ to be very ample in light of Proposition \ref{AbelianProperties} (c). If we pull back along a finite dominant separable morphism $f: X' \to X$ then the polarisation on $X'$ will be given by $f^\ast \mathcal{O}_X(1)$.

By Proposition \ref{Varazroots} it is enough to show that $\mathcal{S}_1(m)$ is almost zero for $m \geq \nu$ with respect to roots of $\mathcal{O}_X(1)$. So let $\eps > 0$.

In order to be in a situation to apply Lemma \ref{KohoHoppingAb} we pull back the whole situation along $N_X^e$ for $e \gg 0$ and then tensor with $\mathcal{O}_X(1)$. Then we may assume that all $\alpha_{i,j}$ are nonzero and that $\deg \mathcal{O}_X(1)/\deg N_X^e < \eps$. We therefore have an inclusion \[H^1(X,{N_X^e}^\ast \mathcal{S}_1(m) \otimes \mathcal{O}_X(1)) \to H^d(X, {N_X^e}^\ast\mathcal{S}_d(m) \otimes \mathcal{O}_X(1)).\] Note that $\mu_{\min}({N_X^e}^\ast \mathcal{S}_d(m) \otimes \mathcal{O}_X(1)) \geq \deg \mathcal{O}_X(1) > 0$.
Again pulling back along $N_X^c$ for $c \gg 0$ we have that $\mu_{\min}({N_X^c}^\ast({N_X^e}^\ast \mathcal{S}_d(m) \otimes \mathcal{O}_X(1))) \geq \deg \mathcal{O}_X(N^{2c})$. Thus the degree (relative to the original $\mathcal{O}_X(1)$) is $\geq N^{2c} \deg \mathcal{O}_X(1) > \deg \omega_X$ for $c$ sufficiently large.
Hence, by Lemma \ref{CohoSerreVanishing} we have that the cohomology group $H^d(X, {N_X^e}^\ast\mathcal{S}_d(m) \otimes \mathcal{O}_X(1))$ vanishes. Consequently, $\mathcal{S}_1(m)$ is almost zero and we have the desired inclusion.
\end{proof}

\begin{Bem}
There is actually no advantage in only working with separable morphisms since for a locally free sheaf $\mathcal{E}$ on an abelian variety one has $\mu_{\min}(\mathcal{E}) = \bar{\mu}_{\min}(\mathcal{E})$ (see \cite[Theorem 2.1 and Remark 2.2]{mehtaramanathanhomogeneous}).
\end{Bem}

\begin{Bem}
Assuming that ample line bundles are almost zero in characteristic zero (with annihilators as roots of the given line bundle)\footnote{We do not know whether this is true but we suspect so.} one can show that the Cohen-Macaulay property is also not necessary for smooth projective surfaces in order to transfer the issue, whether a given cohomology class is almost zero, into a top-dimensional cohomology group. In this case the intermediate cohomology groups are just the first cohomology groups and Kodaira vanishing takes care of the negative twists. Furthermore, $H^1(X, \mathcal{O}_X)$ is almost zero by \cite[Theorem 3.4]{robertssinghannihilators} (provided that ample line bundles are almost zero this also follows by twisting with suitable roots of an ample line bundle).

In general, in characteristic zero, Kodaira's vanishing theorem (see \cite[Remark 7.15]{hartshornealgebraic}) implies that the intermediate cohomology of $\mathcal{O}_X(n)$ vanishes for $n < 0$ (again assuming that $X$ is a smooth projective variety). But one is still left with the problem of annihilating the higher intermediate cohomology groups for positive twists. This question is also related to the property of \emph{almost Cohen-Macaulay} -- see \cite[Definition 1.2 and the discussion thereafter]{robertssinghannihilators}. Almost Cohen-Macaulay does imply the vanishing of these intermediate cohomology groups. But it seems unclear whether it is equivalent to this condition.
\end{Bem}

The following provides examples where graded dagger closure is nontrivial in normal domains of dimension $\geq 2$ containing a field of characteristic zero. For non-normal domains already $IS \cap R$ may be strictly larger than $I$ for an ideal $I \subseteq R$, where $R \subseteq S$ is the normalisation. In positive characteristic this nontriviality is immediate since graded dagger closure contains graded plus closure which can be non-trivial for normal domains in positive characteristic.

\begin{Bsp}
\label{BspDaggerNontrivial}
Let $(X, \mathcal{O}_X(1))$ be a polarised abelian variety of dimension $d$ (e.\,g.\ the Jacobian of a curve of genus $d$). Assume that $\mathcal{O}_X(1)$ is symmetric, generated by global sections and denote its section ring by $R$. Fix homogeneous parameters $x_1, \ldots, x_{d+1}$ in $R_1$. We want to show that $(x_1, \ldots, x_{d+1})^{\dagger \text{GR}}$ is strictly larger than $(x_1, \ldots, x_{d+1})$.

Let \[ 0 \to \mathcal{O}_X(-d-1) \to \bigoplus_{i=1}^{d+1} \mathcal{O}_X(-d) \to \ldots \to \bigoplus_{i=1}^{d+1} \mathcal{O}_X(-1) \to \mathcal{O}_X\to 0\] be the sheafified Koszul complex (see \cite[Remark 1.6.15]{brunsherzog}).
Of course, the Koszul complex on $\Spec R$ is not exact since $x_1, \ldots, x_{d+1}$ do not form a regular sequence but the complex is exact for every localisation at a homogeneous prime $P \neq R_+$ (cf.\ \cite[Proposition 1.6.7]{brunsherzog} and [ibid., Theorem 1.6.16] -- $H_0$ vanishes since the ideal is primary). Hence, it is an exact complex on $\Proj R = X$.

Twisting by $\mathcal{O}_X(d+1)$ one has that the intermediate cohomology vanishes (apply Lemma \ref{KohoHoppingAb}). We therefore have an isomorphism $H^1(X, \mathcal{S}_1(d+1)) \to H^d(X, \mathcal{O}_X)$ (this is surjective since $H^d(X, \mathcal{O}_X(1)) = 0$). 
By virtue of Theorem \ref{AbelianInclusion} we obtain an inclusion $R_{d+1} \subseteq (x_1, \ldots, x_{d+1})^{\dagger \text{GR}}$. Moreover, since $\dim H^d(X, \mathcal{O}_X) = 1$ there are indeed elements in $R_{d+1}$ that are not contained in the ideal.
\end{Bsp}

In order to have an even more explicit example we will consider the case of an elliptic curve.

\begin{Bsp}
Consider $R = k[x,y,z]/(x^3 + y^3 + z^3)$ and assume $\chara k \neq 3$, where $k$ is an algebraically closed field. The corresponding curve $Y = \Proj R$ is elliptic. It is a classical example of tight closure theory that $z^2 \in (x,y)^\ast$ (see e.\,g.\ \cite[Example 1.9]{brennerbarcelona}). It is not contained in the ideal since $z^2 \neq 0$ in $k[z]/(z^3)$. We will show that $z^2 \in (x,y)^{\dagger \text{GR}}$.

The Koszul resolution of $(x,y)$ yields the short exact sequence $0 \to \mathcal{O}_Y(-2) \to \mathcal{O}_Y(-1)^2 \to \mathcal{O}_Y \to 0$ on $Y$. We therefore have $\Syz(x,y)(2) = \mathcal{O}_Y$ and the extension bundle $\mathcal{S}' = \Syz(z^2, x, y)(2)$  corresponding to $\delta(z^2) = \frac{z^2}{xy}$ yields the unique ruled surface $\mathbb{P}(\mathcal{S}'^\vee)$ over $Y$ that is given by the non-split extension $0 \to \mathcal{O}_Y \to \mathcal{S}'^\vee \to \mathcal{O}_Y \to 0$. For the uniqueness see \cite[Theorem 2.15]{hartshornealgebraic} and note that it does not split since $z^2 \notin (x,y)$.

We note that (up to a linear change of variables) this example also occurs in \cite[Example 2.4]{robertssinghannihilators}, where they used a similar technique. But their ring extensions do not stem from any multiplication map since their extensions multiply the degree by $3$ which is not a square.

Let $N =2$ and consider the morphism $N_Y: Y \to Y$ which has degree $N^2 = 4$. Pulling back $\mathcal{O}_Y(1)$ to this new copy of $Y$ we have $N_Y^\ast \mathcal{O}_Y(1) = \mathcal{O}_Y(N^2)$. Taking the section ring $R_1$ induced by $\mathcal{O}_Y(1)$ on this new copy of $Y$ we obtain a ring extension $R = R_0 \subseteq R_1$, where an element of degree $1$ in $R$ maps to an element of degree $4$. Hence, after regrading, $R_1$ is generated by elements of degree $1/4$. Iterating this process we obtain graded ring extensions $R_n \subseteq R_{n+1}$, where every $R_n$ is isomorphic to $R$ and the generators of $R_n$ are in degree $1/4^n$.

We now want to see how $\delta(z^2) = c \in H^1(Y, \mathcal{O}_Y)$ is annihilated by elements of arbitrarily small order in cohomology. So let $Y_n = \Proj R_n$ and pull back the whole situation to $Y_n$. We have $(N^n_Y)^\ast( \Syz(x, y)(2)) = \mathcal{O}_{Y_n}$. As $\mathcal{O}_{Y_n}(1)$ is ample and $Y_n$ is elliptic $H^1(Y_n, \mathcal{O}_{Y_n}(1)) = 0$ (cf.\ \cite[Example IV.1.3.4]{hartshornealgebraic}). In particular, Riemann-Roch yields that $\mathcal{O}_{Y_n}(1)$ has nonzero global sections. Fix any such nonzero global section $s$ and consider the induced morphism $\mathcal{O}_{Y_n} \to \mathcal{O}_{Y_n}(1)$. Taking cohomology we obtain that $s$ annihilates $(N^{n}_Y)^\ast\delta(z^2)$. And since $\mathcal{O}_{Y_n}(N^{2n}) = (N^n_Y)^\ast \mathcal{O}_Y(1)$ we have that $\frac{\deg \mathcal{O}_{Y_n}(1)}{\deg N^n_Y} = \frac{\deg \mathcal{O}_Y(1)}{N^{2n}} = \frac{\deg \mathcal{O}_Y(1)}{4^n}$ which goes to zero as $n$ tends to infinity.
\end{Bsp}

\section{Intermediate cohomology in positive characteristics}
In this section we want to point out that the assertions of \cite[Theorem 2]{brennerlinearfrobenius} can be weakened in so far as one may drop the assumption that $R$ be Cohen-Macaulay and that one may replace $R$ being standard graded with $R$ being a section ring.

The Cohen-Macaulay property is used to ensure that the intermediate cohomology of the twisted structure sheaves vanish. However, the following proposition shows that this may be circumvented.
 
\begin{Prop}
\label{HHVanishing}
Let $(X, \mathcal{O}_X(1))$ be a polarised projective variety over an algebraically closed field of characteristic $p > 0$ and $t_0 \in \mathbb{Z}$. Then there is a finite dominant morphism $f: X' \to X$ of varieties such that the induced maps $H^i(X, \mathcal{O}_X(t)) \to H^i(X', f^\ast \mathcal{O}_X(t))$ are zero for all $t \geq t_0$ and $1 \leq i < \dim X$.
\end{Prop}
\begin{proof}
To begin with one has $F^\ast \mathcal{O}_X(1) = \mathcal{O}_X(p)$, where $F$ denotes the relative Frobenius, so we may assume that $\mathcal{O}_X(1)$ is very ample. Moreover, there is a $t_1$ such that for all $t \geq t_1$ the cohomology groups $H^i(X, \mathcal{O}_X(t))$ vanish for $i > 0$ by virtue of \cite[Proposition III.5.3]{hartshornealgebraic}. By a vanishing theorem of Hochster and Huneke (see \cite[Theorem 1.2]{hochsterhunekeinfinitebig}) there are finite dominant morphisms $f_{i,t}: X_{i,t} \to X$ for $1 \leq i < \dim X$ and $t \in \mathbb{Z}$, where $X_{i,t}$ are varieties, such that the induced morphisms $H^i(X, \mathcal{O}_X(t)) \to H^i(X_{i,t}, f_{i,t}^\ast \mathcal{O}_X(t))$ are zero. Repeatedly applying this theorem for $t_1 \leq t < t_0$ and all $1 \leq i < \dim X$ we deduce that there is a finite dominant morphism with the desired properties.
\end{proof}

\begin{Bem}
If in addition $X$ is Cohen-Macaulay then $H^i(X, \mathcal{O}(-t))$ is dual to $H^{\dim X - i}(X, \mathcal{O}(t) \otimes \omega_X^{\circ})$, where $\omega_X^{\circ}$ is the dualising sheaf of $X$ (cf.\ \cite[Corollary III.7.7]{hartshornealgebraic}). And again by \cite[Proposition III.5.3]{hartshornealgebraic} these groups vanish for $i > 0$ and $t \gg 0$. In particular, there is a finite dominant morphism $f: X' \to X$, where $X'$ is a variety, such that all maps $H^i(X, \mathcal{O}_X(t)) \to H^i(X', f^\ast \mathcal{O}_X(t))$ are zero for $i >0$ and $t \in \mathbb{Z}$.
\end{Bem}

We thus obtain that for an arbitrary polarised projective variety $X$ over an algebraically closed field of characteristic $p > 0$ the assumption that the section ring be Cohen-Macaulay is also not necessary in order to transfer the issue whether an intermediate cohomology class is almost zero into a top-dimensional cohomology class.

To see this fix a resolution of the syzygy bundle in question. Since we only need to look at finitely many twists we may apply Proposition \ref{HHVanishing} to see that all intermediate cohomology of the occurring twisted structure sheaves is mapped to zero along a finite dominant pullback. Hence, if a cohomology class is mapped to zero during ``cohomology hopping'' with respect to a fixed resolution then it is almost zero.
In particular, the assumption that $R$ be Cohen-Macaulay in \cite[Theorem 2]{brennerlinearfrobenius} may be omitted (one can still apply the Serre-Duality argument since $\Proj R$ has a dualising sheaf -- cf. \cite[Proposition III.7.5]{hartshornealgebraic}) and one may also relax the condition that the ring $R$ is standard graded to $R$ being a section ring. This follows since tight closure is also invariant under finite pullbacks.
We cannot weaken the conditions of \cite[Theorem 1]{brennerlinearfrobenius} however. Indeed, Proposition \ref{HHVanishing} makes essential use of passing to finite ring extensions and this may change the Frobenius closure.

\begin{Ko}
Let $(X, \mathcal{O}_X(1))$ be a polarised projective variety over an algebraically closed field $k$ of positive characteristic. Then any semiample line bundle $\mathcal{L}$ on $X$ is almost zero and we can choose the annihilators as suitable roots of $\mathcal{O}_X(1)$.
\end{Ko}
\begin{proof}
Let $\eps > 0$. Applying Proposition \ref{PullbackSemiampleGloballygenerated} we may assume that $\mathcal{L}$ is globally generated. Applying \cite[Lemma 3.2]{brennerstaeblerdaggersolid} and Proposition \ref{PullbackSemiampleGloballygenerated} we may assume that there is a finite dominant morphism $\varphi: X' \to X$ of varieties such that there is a globally generated line bundle $\mathcal{M}$ on $X'$ which is an $n$th root of $\mathcal{O}_X(1)$ such that $\frac{\deg \mathcal{M}}{\deg \varphi} < \eps$ (simply choose $n$ sufficiently large). Since $\mathcal{M}$ is globally generated we have an induced morphism $\varphi^\ast \mathcal{L} \to \varphi^\ast \mathcal{L} \otimes \mathcal{M}$ and $\varphi^\ast \mathcal{L} \otimes \mathcal{M}$ is ample since $\varphi^\ast \mathcal{L}$ is globally generated (see \cite[Ex.\ II.7.5 (a)]{hartshornealgebraic}). By Proposition \ref{HHVanishing} there is a pullback along a finite dominant morphism which annihilates $H^1(X', \varphi^\ast \mathcal{L} \otimes \mathcal{M})$. In particular, $\mathcal{L}$ is almost zero and we can choose the annihilators as suitable roots of $\mathcal{O}_X(1)$.
\end{proof}

\begin{Bem}
It is actually true that some pullback along a finite dominant morphism already annihilates the cohomology of a semiample line bundle in positive characteristic -- see \cite[Proposition 6.2]{bhattderivedsplinters}.
\end{Bem}

\bibliography{bibliothek.bib}
\bibliographystyle{amsplain}

\end{document}